\documentclass[12pt,english]{article}
\usepackage[T1]{fontenc}
\usepackage[latin9]{inputenc}
\usepackage{geometry}
\usepackage{mathtools}
\usepackage{amsmath}
\usepackage{amsthm}
\usepackage{amssymb}

\makeatletter
\theoremstyle{plain}
\newtheorem{thm}{\protect\theoremname}
  \theoremstyle{remark}
  \newtheorem{rem}[thm]{\protect\remarkname}
  \theoremstyle{plain}
  \newtheorem{prop}[thm]{\protect\propositionname}
  \theoremstyle{plain}
  \newtheorem{lem}[thm]{\protect\lemmaname}
  \theoremstyle{plain}
  \newtheorem{cor}[thm]{\protect\corollaryname}

\usepackage{bm}

\usepackage{tikz}
\usetikzlibrary{arrows}
\usepackage{hyperref}
\usepackage{microtype}

\usepackage{colortbl}
\usepackage{xparse}
\usepackage{caption}

\usepackage{titlesec}
\titleformat{\section}
  {\normalfont\large\bfseries}{\thesection.}{.8ex}{} 
\titleformat{\subsection}
  {\normalfont\bfseries}{\thesubsection.}{.8ex}{} 

\usetikzlibrary{shapes}

\tikzstyle{pathdefault}=[draw, line width=1, solid, color=black]
\tikzstyle{nodedefault}=[circle, inner sep=1.5, fill=black]
\tikzstyle{empty}=[]
\tikzstyle{nodeellipsis}=[circle, inner sep=0.5, fill=black]
\tikzstyle{pathcolor1}=[draw, line width=1.3, densely dashed, color=red]
\tikzstyle{pathcolor2}=[draw, line width=1.6, densely dotted, color=blue]
\tikzstyle{pathcolorlight}=[draw, line width=1, dotted, color=lightgray]

\tikzstyle{arbpathcolor0}=[line width=1, dashdotted, color=black]
\tikzstyle{arbpathcolor1}=[line width=1, densely dashed, color=red]
\tikzstyle{arbpathdefault}=[line width=1, densely dotted, color=blue]

\newcounter{id}
\newcommand{\drawlinedotswithstyle}[4]{
 \def\x{{#3}}
 \def\y{{#4}}
 \tikzstyle{thispathstyle}=[#1]
 \tikzstyle{thisnodestyle}=[#2]
 \setcounter{id}{-1} 
 \foreach \j in {#3}{\stepcounter{id}} 
 \foreach \i in {1,...,\the\value{id}}{  
  \path[thispathstyle] (\x[\i],\y[\i]) --(\x[\i-1],\y[\i-1]); 
 }
 \foreach \i in {1,...,\the\value{id}}{  
  \node[thisnodestyle] at (\x[\i],\y[\i]) {}; 
 }
 \node[thisnodestyle] at (\x[0],\y[0]) {}; 
}

\DeclareDocumentCommand{\drawlinedots}{ O{pathdefault} O{nodedefault} m m}{\drawlinedotswithstyle{#1}{#2}{#3}{#4}}

\let\originalleft\left
\let\originalright\right
\renewcommand{\left}{\mathopen{}\mathclose\bgroup\originalleft}
\renewcommand{\right}{\aftergroup\egroup\originalright}

\definecolor{mhcblue}{HTML}{0077CC} 
\definecolor{davidsonred}{HTML}{AC1A2F} 

\definecolor{green}{RGB}{0, 180, 0}
\definecolor{yellow}{RGB}{180, 180, 0}

\makeatother

\usepackage{babel}
  \providecommand{\corollaryname}{Corollary}
  \providecommand{\lemmaname}{Lemma}
  \providecommand{\propositionname}{Proposition}
  \providecommand{\remarkname}{Remark}
\providecommand{\theoremname}{Theorem}

\begin{document}
\global\long\def\exc{\operatorname{exc}}
\global\long\def\des{\operatorname{des}}
\global\long\def\Cpk{\operatorname{Cpk}}
\global\long\def\Cval{\operatorname{Cval}}
\global\long\def\Cdasc{\operatorname{Cdasc}}
\global\long\def\Cddes{\operatorname{Cddes}}
\global\long\def\Fix{\operatorname{Fix}}
\global\long\def\Exc{\operatorname{Exc}}
\global\long\def\cpk{\operatorname{cpk}}
\global\long\def\cval{\operatorname{cval}}
\global\long\def\cddes{\operatorname{cddes}}
\global\long\def\cdasc{\operatorname{cdasc}}
\global\long\def\Orb{\operatorname{Orb}}
\global\long\def\fix{\operatorname{fix}}
\global\long\def\pk{\operatorname{pk}}

\title{On the joint distribution of cyclic valleys and excedances over conjugacy
classes of $\mathfrak{S}_{n}$}

\author{M.\ Crossan Cooper\qquad{}William S.\ Jones\qquad{}Yan Zhuang\\
Department of Mathematics and Computer Science\\
Davidson College\texttt{}~\\
\texttt{\{crcooper, wijones, yazhuang\}@davidson.edu}}
\maketitle
\begin{abstract}
We derive a formula expressing the joint distribution of the cyclic
valley number and excedance number statistics over a fixed conjugacy
class of the symmetric group in terms of Eulerian polynomials. Our
proof uses a slight extension of Sun and Wang's cyclic valley-hopping
action as well as a formula of Brenti. Along the way, we give a new
proof for the $\gamma$-positivity of the excedance number distribution
over any fixed conjugacy class along with a combinatorial interpretation
of the $\gamma$-coefficients.
\end{abstract}
\textbf{\small{}Keywords:}{\small{} permutation statistics, excedances,
cyclic valleys, Eulerian polynomials, $\gamma$-positivity, modified
Foata\textendash Strehl action}{\let\thefootnote\relax}

\section{Introduction}

Let $\pi=\pi(1)\pi(2)\cdots\pi(n)$ be a permutation in the symmetric
group $\mathfrak{S}_{n}$. We say that $i\in[n-1]$ is a \textit{descent}
of $\pi$ if $\pi(i)>\pi(i+1)$ and that $i\in[n]$ is an \textit{excedance}
of $\pi$ if $i<\pi(i)$.\footnote{The set $[n]$ is defined by $[n]\coloneqq\{1,2,\dots,n\}$.}
We let $\des(\pi)$ denote the number of descents of $\pi$ and $\exc(\pi)$
the number of excedances of $\pi$. For example, if $\pi=371896542$,
then the descents of $\pi$ are 2, 5, 6, 7, and 8 whereas the excedances
of $\pi$ are 1, 2, 4, and 5; thus $\des(\pi)=5$ and $\exc(\pi)=4$.
It is well known that the descent number $\des$ and the excedance
number $\exc$ have the same distribution over $\mathfrak{S}_{n}$,
that is, the number of permutations in $\mathfrak{S}_{n}$ with exactly
$k$ descents is equal to the number of permutations in $\mathfrak{S}_{n}$
with exactly $k$ excedances.

Given a polynomial $f$ in the variable $t$, we say that $f$ is
\textit{$\gamma$-positive} with center of symmetry $n/2$ if we can
write
\[
f(t)=\sum_{i=0}^{\left\lfloor n/2\right\rfloor }\gamma_{i}t^{i}(1+t)^{n-2i}
\]
for some non-negative integers $\gamma_{i}$ (called the $\gamma$-coefficients
of $f$). If a polynomial is $\gamma$-positive, then its sequence
of coefficients is symmetric and unimodal. The prototypical example
of a family of $\gamma$-positive polynomials are the Eulerian polynomials
$\{A_{n}(t)\}_{n\geq0}$ defined by
\[
A_{n}(t)\coloneqq\sum_{\pi\in\mathfrak{S}_{n}}t^{\des(\pi)+1}=\sum_{\pi\in\mathfrak{S}_{n}}t^{\exc(\pi)+1}
\]
for $n\geq1$ and by $A_{0}(t)\coloneqq1$. The $n$th Eulerian polynomial
encodes the distribution of the descent number (equivalently, the
excedance number) over $\mathfrak{S}_{n}$. The $\gamma$-positivity
of Eulerian polynomials was proven by Foata and Sch{\"u}tzenberger
\cite{Foata1970} in 1970, long before the term ``$\gamma$-positivity''
was coined, but the general notion of $\gamma$-positivity has emerged
as a powerful way to prove unimodality results and has connections
to many facets of enumerative, algebraic, and geometric combinatorics.
See \cite{Athanasiadis2017} for a comprehensive survey on this topic.

The \textit{cycle type} of a permutation $\pi$ is a partition of
$n$ encoding the number of cycles of $\pi$ of each size. Continuing
with the earlier example, the permutation $\pi=371896542$ in one-line
notation can be written as $\pi=(3,1)(6)(8,4)(9,2,7,5)$ in cycle
notation, which has cycle type $\lambda=(1,2,2,4)$. Conjugacy classes
of the symmetric group $\mathfrak{S}_{n}$ correspond to sets of permutations
with a fixed cycle type, and one may investigate distributions of
permutation statistics over conjugacy classes. Perhaps the most famous
result in this domain is by Gessel and Reutenauer \cite{Gessel1993},
who proved that the number of permutations in a prescribed conjugacy
class with a prescribed descent set is equal to the scalar product
of a ribbon Schur function and a Lyndon symmetric function (equivalently,
the scalar product of the characters of a Foulkes representation and
a Lie representation). 

We say that $i\in[n]$ is a \textit{cyclic valley} of $\pi\in\mathfrak{S}_{n}$
if $\pi^{-1}(i)>i<\pi(i)$, and we let $\cval(\pi)$ denote the number
of cyclic valleys of $\pi$. In this paper, we will study the polynomials
\begin{gather*}
E_{\lambda}(t)\coloneqq\sum_{\pi\in\mathfrak{S}_{n}(\lambda)}t^{\exc(\pi)},\qquad E_{\lambda}^{\cval}(t)\coloneqq\sum_{\pi\in\mathfrak{S}_{n}(\lambda)}t^{\cval(\pi)},
\end{gather*}
and 
\[
E_{\lambda}^{(\cval,\exc)}(s,t)\coloneqq\sum_{\pi\in\mathfrak{S}_{n}(\lambda)}s^{\cval(\pi)}t^{\exc(\pi)},
\]
where $\mathfrak{S}_{n}(\lambda)$ is the conjugacy class of $\mathfrak{S}_{n}$
consisting of all permutations with cycle type $\lambda$. While the
polynomials $E_{\lambda}^{\cval}(t)$ and $E_{\lambda}^{(\cval,\exc)}(s,t)$
appear to be new, the $E_{\lambda}(t)$ were studied earlier by Brenti.
For a partition $\lambda=(1^{m_{1}}2^{m_{2}}\cdots)$ of $n$\textemdash that
is, a partition with $m_{i}$ parts of size $i$ for each $i$\textemdash Brenti
\cite[Theorem 3.1]{Brenti1993} proved the formula
\begin{equation}
E_{\lambda}(t)=\frac{n!}{z_{\lambda}}\prod_{i\geq1}\left[\frac{A_{i-1}(t)}{(i-1)!}\right]^{m_{i}}\label{e-brenti}
\end{equation}
where the constant $z_{\lambda}$ is defined by $z_{\lambda}\coloneqq\prod_{i\geq1}i^{m_{i}}m_{i}!$.
Since Eulerian polynomials are $\gamma$-positive and products of
$\gamma$-positive polynomials are $\gamma$-positive \cite[Observation 4.1]{Petersen2015},
Brenti's formula implies that the polynomials $E_{\lambda}(t)$ are
$\gamma$-positive as well.

Our main result is the following formula:
\begin{thm}
\label{t-mainthm} Let $\lambda=(1^{m_{1}}2^{m_{2}}\cdots)$ be a
partition of $n$. Then
\[
E_{\lambda}^{(\cval,\exc)}(s,t)=\frac{n!}{z_{\lambda}}\left(\frac{1+u}{1+uv}\right)^{n-m_{1}}\prod_{i\geq1}\left[\frac{A_{i-1}(v)}{(i-1)!}\right]^{m_{i}}
\]
where $u=\frac{1+t^{2}-2st-(1-t)\sqrt{(1+t)^{2}-4st}}{2(1-s)t}$ and
$v=\frac{(1+t)^{2}-2st-(1+t)\sqrt{(1+t)^{2}-4st}}{2st}.$
\end{thm}

This formula allows one to compute the joint distribution of the statistics
$\cval$ and $\exc$ over any fixed conjugacy class directly from
Eulerian polynomials. For example, take $\lambda=(1,5,5)$. Then 
\begin{align*}
E_{\lambda}^{(\cval,\exc)}(s,t) & =\frac{11!}{1^{1}1!\cdot5^{2}2!}\left(\frac{1+u}{1+uv}\right)^{10}\frac{A_{0}(v)A_{4}(v)^{2}}{0!4!^{2}}\\
 & =1386\left(\frac{1+u}{1+uv}\right)^{10}(v+11v^{2}+11v^{3}+v^{4})^{2}\\
 & =(1386t^{2}+8316t^{3}+20790t^{4}+27720t^{5}+20790t^{6}+8316t^{7}+1386t^{8})s^{2}\\
 & \qquad\qquad\qquad\;\;\,+(22176t^{3}+88704t^{4}+133056t^{5}+88704t^{6}+22176t^{7})s^{3}\\
 & \qquad\qquad\qquad\qquad\qquad\qquad\qquad\quad\;\;+(88704t^{4}+177408t^{5}+88704t^{6})s^{4}.
\end{align*}
(The last equality can be verified using a computer algebra system
such as Maple.)
\begin{rem}
\label{r-unisym} Notice that, in the above example, the coefficient
of $s^{i}$ for each $i$ is a unimodal and symmetric polynomial in
$t$; we will see that this is true for all $E_{\lambda}^{(\cval,\exc)}(s,t)$.
\end{rem}

To prove Theorem \ref{t-mainthm}, we will first extend a group action
of Sun and Wang \cite{Sun2014} defined on derangements (permutations
without fixed points) to all permutations; we call this action ``cyclic
valley-hopping''. We will then prove a technical lemma (Lemma \ref{l-excorb})
using cyclic valley-hopping which is in turn utilized to derive a
formula expressing the polynomial $E_{\lambda}^{(\cval,\exc)}(s,t)$
in terms of the polynomial $E_{\lambda}(t)$. We then combine this
formula with Brenti's formula (\ref{e-brenti}) to yield Theorem \ref{t-mainthm}
and derive a similar result (Theorem \ref{t-mainresult2}) expressing
the polynomial $E_{\lambda}^{\cval}(t)$ in terms of Eulerian polynomials.
Along the way, we use Lemma \ref{l-excorb} to obtain an alternative
proof for the $\gamma$-positivity of the polynomials $E_{\lambda}(t)$
in a way which yields a combinatorial interpretation for their $\gamma$-coefficients
and which will explain our observation in Remark \ref{r-unisym}.

\section{Preliminaries}

\subsection{Permutation statistics}

We begin with a brief discussion of several more permutation statistics
which will arise in this paper. Given a permutation $\pi=\pi(1)\pi(2)\cdots\pi(n)$
in $\mathfrak{S}_{n}$, we say that $\pi(i)$ is:
\begin{itemize}
\item a \textit{valley }of $\pi$ if $\pi(i-1)>\pi(i)<\pi(i+1)$;
\item a \textit{peak} of $\pi$ if $\pi(i-1)<\pi(i)>\pi(i+1)$;
\item a \textit{double ascent} of $\pi$ if $\pi(i-1)<\pi(i)<\pi(i+1)$;
\item a \textit{double descent} of $\pi$ if $\pi(i-1)>\pi(i)>\pi(i+1)$.
\end{itemize}
In our work, we will be more concerned with cyclic analogues of these
notions which have also been well-studied, e.g., in \cite{Shin2012,Sun2014,Tirrell2018,Zeng1993}.
We have already defined excedances and cyclic valleys of $\pi$. We
say that $i\in[n]$ is:
\begin{itemize}
\item a \textit{cyclic peak} of $\pi$ if $\pi^{-1}(i)<i>\pi(i)$;
\item a \textit{cyclic double ascent} of $\pi$ if $\pi^{-1}(i)<i<\pi(i)$;
\item a \textit{cyclic double descent} of $\pi$ if $\pi^{-1}(i)>i>\pi(i)$;
\item a \textit{fixed point} of $\pi$ if $\pi(i)=i$.
\end{itemize}
It is clear that every cycle of size one is a fixed point, and that
in any cycle of size at least two, the first letter is a cyclic peak
and the last letter is either a cyclic valley or a cyclic double ascent.

Define $\Exc(\pi)$, $\Cval(\pi)$, $\Cpk(\pi)$, $\Cdasc(\pi)$,
$\Cddes(\pi)$, and $\Fix(\pi)$ to be the set of excedances, cyclic
valleys, cyclic peaks, cyclic double ascents, cyclic double descents,
and fixed points, respectively. Moreover, let $\cpk(\pi)\coloneqq\left|\Cpk(\pi)\right|$,
$\cdasc(\pi)\coloneqq\left|\Cdasc(\pi)\right|$, $\cddes(\pi)\coloneqq\left|\Cddes(\pi)\right|$,
and $\fix(\pi)\coloneqq\left|\Fix(\pi)\right|$. 

As an example, take $\pi=(5,2,1)(6)(8)(11,9,10,4,3,7)$. Here $\Exc(\pi)=\{1,3,7,9\}$,
$\Cval(\pi)=\{1,3,9\}$, $\Cpk(\pi)=\{5,10,11\}$, $\Cdasc(\pi)=\{7\}$,
$\Cddes(\pi)=\{2,4\}$, and $\Fix(\pi)=\{6,8\}$. Thus $\exc(\pi)=4$,
$\cval(\pi)=3$, $\cpk(\pi)=3$, $\cdasc(\pi)=1$, $\cddes(\pi)=2$,
and $\fix(\pi)=2$. 

It is clear from the definitions that every letter of a permutation
is either a cyclic valley, cyclic peak, cyclic double ascent, cyclic
double descent, or fixed point. Thus, we have 
\[
\Cval(\pi)\cup\Cpk(\pi)\cup\Cdasc(\pi)\cup\Cddes(\pi)\cup\Fix(\pi)=[n]
\]
and 
\begin{equation}
\cval(\pi)+\cpk(\pi)+\cdasc(\pi)+\cddes(\pi)+\fix(\pi)=n\label{e-ndecomp}
\end{equation}
for any $\pi\in\mathfrak{S}_{n}$. It is also clear that the excedances
of a permutation are precisely its cyclic valleys and cyclic double
ascents, that is, 

\begin{equation}
\Cval(\pi)\cup\Cdasc(\pi)=\Exc(\pi)\label{e-excdecompset}
\end{equation}
and 
\begin{equation}
\cval(\pi)+\cdasc(\pi)=\exc(\pi)\label{e-excdecomp}
\end{equation}
for all $\pi\in\mathfrak{S}_{n}$. Finally, it is not difficult to
see that, for any $\pi\in\mathfrak{S}_{n}$, the sets $\Cpk(\pi)$
and $\Cval(\pi)$ are in bijection. Hence, we have
\begin{equation}
\cpk(\pi)=\cval(\pi).\label{e-cpkcval}
\end{equation}

Before continuing, we give a couple remarks on cycle notation. When
writing permutations in cycle notation, we adopt the convention of
writing each cycle with its largest letter in the first position,
and writing the cycles from left-to-right in increasing order of their
largest letters. (This convention is sometimes called \textit{canonical
cycle representation}.) For example, the permutation $\pi=649237185$
in one-line notation is written as $\pi=(42)(716)(8)(953)$ in cycle
notation.

We will make use of a map called Foata's ``transformation fondamentale'';
this map $o\colon\mathfrak{S}_{n}\rightarrow\mathfrak{S}_{n}$ is
defined by taking as input a permutation in canonical cycle representation
and the output is the permutation in one-line notation obtained by
erasing the parentheses. Continuing the example with $\pi=(42)(716)(8)(953)$,
we have $o(\pi)=427168953$. It is easy to see that this map is a
bijection; we can recover the cycles of $o^{-1}(\pi)$ from a permutation
$\pi$ by noting the \textit{left-to-right maxima} of $\pi$: letters
$\pi(i)$ for which $\pi(j)<\pi(i)$ for all $1\leq j<i$.

\subsection{Cyclic valley-hopping}

Our remaining goal in this preliminary section is to define a group
action on $\mathfrak{S}_{n}$ induced by involutions which toggle
between cyclic double ascents and cyclic double descents. Before we
define this group action, it will be convenient to first define two
related group actions. Fix a permutation $\pi\in\mathfrak{S}_{n}$
and a letter $x\in[n]$. We may write $\pi=w_{1}w_{2}xw_{4}w_{5}$
where $w_{2}$ is the maximal consecutive subword immediately to the
left of $x$ whose letters are all smaller than $x$, and $w_{4}$
is the maximal consecutive subword immediately to the right of $x$
whose letters are all smaller than $x$; this decomposition is called
the $x$\textit{-factorization} of $\pi$. For example, if $\pi=834279156$
and $x=7$, then $\pi$ is the concatenation of $w_{1}=8$, $w_{2}=342$,
$x=7$, the empty word $w_{4}$, and $w_{5}=9156$.

Define $\varphi_{x}\colon\mathfrak{S}_{n}\rightarrow\mathfrak{S}_{n}$
by 
\[
\varphi_{x}(\pi)\coloneqq\begin{cases}
w_{1}w_{4}xw_{2}w_{5}, & \mbox{if }x\mbox{ is a double ascent or double descent of \ensuremath{\pi},}\\
\pi, & \mbox{if }x\mbox{ is a peak or valley of \ensuremath{\pi}.}
\end{cases}
\]
(Here, we are using the conventions $\pi(0)=\pi(n+1)=\infty$.) Equivalently,
$\varphi_{x}(\pi)=w_{1}w_{4}xw_{2}w_{5}$ if exactly one of $w_{2}$
and $w_{4}$ is nonempty, and $\varphi_{x}(\pi)=\pi$ otherwise. It
is easy to see that $\varphi_{x}$ is an involution, and that for
all $x,y\in[n]$, the involutions $\varphi_{x}$ and $\varphi_{y}$
commute with each other. Given a subset $S\subseteq[n]$, we define
$\varphi_{S}\colon\mathfrak{S}_{n}\rightarrow\mathfrak{S}_{n}$ by
$\varphi_{S}\coloneqq\prod_{x\in S}\varphi_{x}$. For example, given
$\pi=834279156$ and $S=\{6,7,8\}$, we have $\varphi_{S}(\pi)=734289615$;
see Figure 1. The involutions $\{\varphi_{S}\}_{S\subseteq[n]}$ define
a $\mathbb{Z}_{2}^{n}$-action on $\mathfrak{S}_{n}$ which is commonly
known as the \textit{modified Foata\textendash Strehl action} or \textit{valley-hopping}.
This action is based on a classical group action of Foata and Strehl
\cite{Foata1974}, was introduced by Shapiro, Woan, and Getu \cite{Shapiro1983},
and later rediscovered by Br\"and\'en \cite{Braenden2008}.
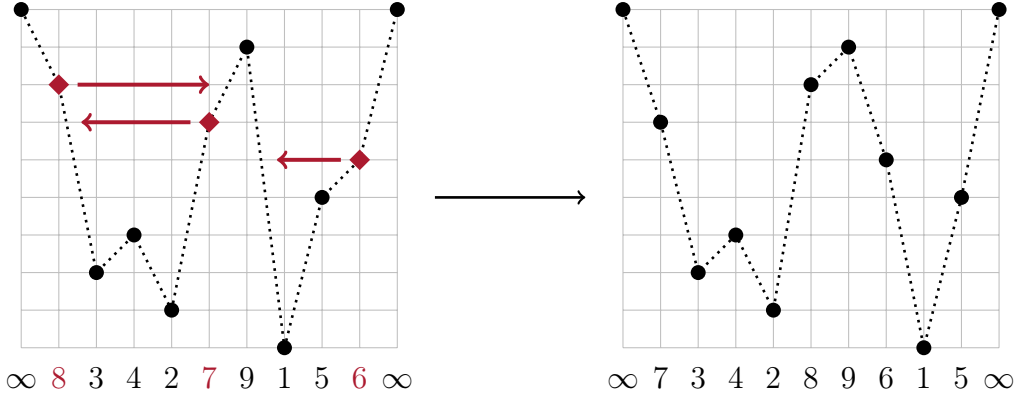
\begin{figure}
\begin{center}
\begin{tikzpicture}[scale=0.5] 	
\draw[step=1,lightgray,thin] (0,1) grid (10,10); 
	\tikzstyle{ridge}=[draw, line width=1, dotted, color=black] 
	\path[ridge] (0,10)--(1,8)--(2,3)--(3,4)--(4,2)--(5,7)--(6,9)--(7,1)--(8,5)--(9,6)--(10,10); 
	\tikzstyle{node0}=[circle, inner sep=2, fill=black] 
	\tikzstyle{node1}=[rectangle, inner sep=3, fill=mhcblue] 
	\tikzstyle{node2}=[diamond, inner sep=2, fill=davidsonred] 
	\node[node0] at (0,10) {}; 
	\node[node2] at (1,8) {}; 
	\node[node0] at (2,3) {}; 
	\node[node0] at (3,4) {}; 
	\node[node0] at (4,2) {}; 
	\node[node2] at (5,7) {}; 
	\node[node0] at (6,9) {}; 
	\node[node0] at (7,1) {}; 
	\node[node0] at (8,5) {}; 
	\node[node2] at (9,6) {}; 
	\node[node0] at (10,10) {}; 
	\tikzstyle{hop1}=[draw, line width = 1.5, color=davidsonred,->]
	\tikzstyle{hop2}=[draw, line width = 1.5, color=davidsonred,<-] 
	\path[hop1] (1.5,8)--(5,8);
	\path[hop2] (1.6,7)--(4.5,7);
	\path[hop2] (6.8,6)--(8.5,6); 
	\tikzstyle{pi}=[above=-5] 
	\node[pi] at (0,0) {$\infty$}; 
	\node[pi, color=davidsonred] at (1,0) {8}; 
	\node[pi] at (2,0) {3}; 
	\node[pi] at (3,0) {4}; 
	\node[pi] at (4,0) {2}; 
	\node[pi, color=davidsonred] at (5,0) {7}; 
	\node[pi] at (6,0) {9}; 
	\node[pi] at (7,0) {1}; 
	\node[pi] at (8,0) {5}; 
	\node[pi, color=davidsonred] at (9,0) {6}; 
	\node[pi] at (10,0) {$\infty$}; 
	\path[draw,line width=1,->] (11,5)--(15,5); 
	\begin{scope}[shift={(16,0)}] 
	\draw[step=1,lightgray,thin] (0,1) grid (10,10); 
	\path[ridge] (0,10)--(1,7)--(2,3)--(3,4)--(4,2)--(5,8)--(6,9)--(7,6)--(8,1)--(9,5)--(10,10); 
	\node[node0] at (0,10) {}; 
	\node[node0] at (1,7) {}; 
	\node[node0] at (2,3) {}; 
	\node[node0] at (3,4) {}; 
	\node[node0] at (4,2) {}; 
	\node[node0] at (5,8) {}; 
	\node[node0] at (6,9) {}; 
	\node[node0] at (7,6) {}; 
	\node[node0] at (8,1) {}; 
	\node[node0] at (9,5) {}; 
	\node[node0] at (10,10) {}; 
	\node[pi] at (0,0) {$\infty$}; 
	\node[pi] at (1,0) {7}; 
	\node[pi] at (2,0) {3}; 
	\node[pi] at (3,0) {4}; 
	\node[pi] at (4,0) {2}; 
	\node[pi] at (5,0) {8}; 
	\node[pi] at (6,0) {9};
	\node[pi] at (7,0) {6}; 
	\node[pi] at (8,0) {1}; 
	\node[pi] at (9,0) {5}; 
	\node[pi] at (10,0) {$\infty$}; 
	\end{scope}
\end{tikzpicture}
\end{center}

\caption{Valley-hopping on $\pi=834279156$ with $S=\{6,7,8\}$ yields $\varphi_{S}(\pi)=734289615$}
\end{figure}

Next, we define a group action due to Sun and Wang \cite{Sun2014}
which is an analogue of valley-hopping for derangements in cycle notation.
Let $\mathfrak{D}_{n}$ be the set of derangements of length $n$.
Define $\theta_{x}\colon\mathfrak{D}_{n}\rightarrow\mathfrak{D}_{n}$
by $\theta_{x}(\pi)\coloneqq o^{-1}(\varphi_{x}(o(\pi)))$, where
we treat the $0$th letter of $o(\pi)$ as 0 and the $(n+1)$th letter
as $\infty$. As with the functions $\varphi_{x}$, the functions
$\theta_{x}$ are involutions that commute with each other. Similarly,
for a subset $S\subseteq[n]$, define $\theta_{S}\colon\mathfrak{D}_{n}\rightarrow\mathfrak{D}_{n}$
by $\theta_{S}\coloneqq\prod_{x\in S}\theta_{x}$. Then Sun and Wang's
\textit{cyclic modified Foata\textendash Strehl action} is the $\mathbb{Z}_{2}^{n}$-action
defined by the involutions $\theta_{S}$.

Sun and Wang's action can easily be extended to all permutations;
simply define $\psi_{x}\colon\mathfrak{S}_{n}\rightarrow\mathfrak{S}_{n}$
by 
\[
\psi_{x}(\pi)\coloneqq\begin{cases}
o^{-1}(\varphi_{x}(o(\pi))), & \mbox{if }x\mbox{ is not a fixed point of \ensuremath{\pi},}\\
\pi, & \mbox{if }x\mbox{ is a fixed point of }\pi,
\end{cases}
\]
where, as before, we treat the $0$th letter of $o(\pi)$ as 0 and
the $(n+1)$th letter as $\infty$. Given a subset $S\subseteq[n]$,
define $\psi_{S}\colon\mathfrak{S}_{n}\rightarrow\mathfrak{S}_{n}$
by $\psi_{S}\coloneqq\prod_{x\in S}\psi_{x}$. For example, given
$\pi=(523)(8)(97641)$ and $S=\{3,7\}$, we have $\psi_{S}(\pi)=(532)(8)(96417)$;
see Figure 2. In what follows, we will call the $\mathbb{Z}_{2}^{n}$-action
defined by the involutions $\{\psi_{S}\}_{S\subseteq[n]}$ \textit{cyclic
valley-hopping}.
\begin{figure}
\begin{center}
\begin{tikzpicture}[scale=0.5] 	
\draw[step=1,lightgray,thin] (0,0) grid (10,10); 
	\tikzstyle{ridge}=[draw, line width=1, dotted, color=black] 
	\path[ridge] (0,0)--(1,5)--(2,2)--(3,3)--(4,8)--(5,9)--(6,7)--(7,6)--(8,4)--(9,1)--(10,10); 
	\tikzstyle{node0}=[circle, inner sep=2, fill=black] 
	\tikzstyle{node1}=[rectangle, inner sep=3, fill=mhcblue] 
	\tikzstyle{node2}=[diamond, inner sep=2, fill=davidsonred] 
	\node[node0] at (0,0) {}; 
	\node[node0] at (1,5) {}; 
	\node[node0] at (2,2) {}; 
	\node[node2] at (3,3) {}; 
	\node[node0] at (4,8) {}; 
	\node[node0] at (5,9) {}; 
	\node[node2] at (6,7) {}; 
	\node[node0] at (7,6) {}; 
	\node[node0] at (8,4) {}; 
	\node[node0] at (9,1) {}; 
	\node[node0] at (10,10) {}; 
	\tikzstyle{hop1}=[draw, line width = 1.5, color=davidsonred,->]
	\tikzstyle{hop2}=[draw, line width = 1.5, color=davidsonred,<-] 
	\path[hop2] (2,3)--(2.5,3);
	\path[hop1] (6.5,7)--(9.1,7);
	\tikzstyle{pi}=[above=-20] 
	\node[pi] at (0,0) {0}; 
	\node[pi] at (1,0) {5}; 
	\node[pi] at (2,0) {2}; 
	\node[pi, color=davidsonred] at (3,0) {3}; 
	\node[pi] at (4,0) {8}; 
	\node[pi] at (5,0) {9}; 
	\node[pi, color=davidsonred] at (6,0) {7}; 
	\node[pi] at (7,0) {6}; 
	\node[pi] at (8,0) {4}; 
	\node[pi] at (9,0) {1}; 
	\node[pi] at (10,0) {$\infty$}; 
	\path[draw,line width=1,->] (11,5)--(15,5); 
	\begin{scope}[shift={(16,0)}] 
	\draw[step=1,lightgray,thin] (0,0) grid (10,10); 
	\path[ridge] (0,0)--(1,5)--(2,3)--(3,2)--(4,8)--(5,9)--(6,6)--(7,4)--(8,1)--(9,7)--(10,10); 
	\node[node0] at (0,0) {}; 
	\node[node0] at (1,5) {}; 
	\node[node0] at (2,3) {}; 
	\node[node0] at (3,2) {}; 
	\node[node0] at (4,8) {}; 
	\node[node0] at (5,9) {}; 
	\node[node0] at (6,6) {}; 
	\node[node0] at (7,4) {}; 
	\node[node0] at (8,1) {}; 
	\node[node0] at (9,7) {}; 
	\node[node0] at (10,10) {}; 
	\node[pi] at (0,0) {0}; 
	\node[pi] at (1,0) {5}; 
	\node[pi] at (2,0) {3}; 
	\node[pi] at (3,0) {2}; 
	\node[pi] at (4,0) {8}; 
	\node[pi] at (5,0) {9}; 
	\node[pi] at (6,0) {6};
	\node[pi] at (7,0) {4}; 
	\node[pi] at (8,0) {1}; 
	\node[pi] at (9,0) {7}; 
	\node[pi] at (10,0) {$\infty$}; 
	\end{scope}
\end{tikzpicture}
\end{center}

\caption{Cyclic valley-hopping on $\pi=(523)(8)(97641)$ with $S=\{3,7\}$
yields $\psi_{S}(\pi)=(532)(8)(96417)$}
\end{figure}
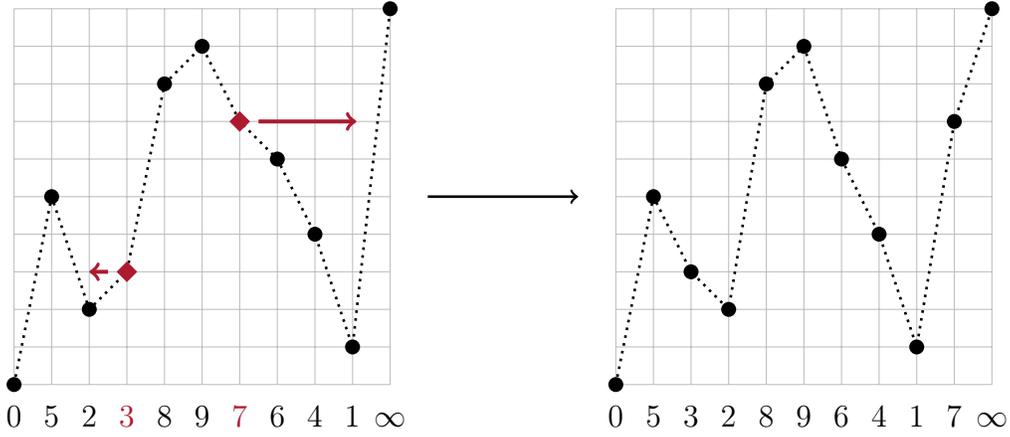

We omit the proof of the next proposition, which describes the cyclic
valleys, cyclic peaks, cyclic double ascents, cyclic double descents,
and fixed points of $\psi_{S}(\pi)$ in terms of those of $\pi$.
The takeaway is that cyclic valley-hopping does not affect cyclic
valleys, cyclic peaks, and fixed points, but toggles between cyclic
double ascents and cyclic double descents.
\begin{prop}
\label{p-statsets} For any $S\subseteq[n]$ and $\pi\in\mathfrak{S}_{n}$,
we have:
\begin{itemize}
\item [\normalfont{(a)}]$\Cval(\psi_{S}(\pi))=\Cval(\pi)$;
\item [\normalfont{(b)}]$\Cpk(\psi_{S}(\pi))=\Cpk(\pi)$;
\item [\normalfont{(c)}]$\Cdasc(\psi_{S}(\pi))=(\Cdasc(\pi)\backslash S)\cup(S\cap\Cddes(\pi))$;
\item [\normalfont{(d)}]$\Cddes(\psi_{S}(\pi))=(\Cddes(\pi)\backslash S)\cup(S\cap\Cdasc(\pi))$;
\item [\normalfont{(e)}]$\Fix(\psi_{S}(\pi))=\Fix(\pi)$.
\end{itemize}
\end{prop}

We say that a subset $\Pi\subseteq\mathfrak{S}_{n}$ is \textit{invariant
under cyclic valley-hopping} if for every $S\subseteq[n]$ and permutation
$\pi\in\Pi$, we have $\psi_{S}(\pi)\in\Pi$ (equivalently, if $\Pi$
is a disjoint union of orbits of the cyclic valley-hopping action).
\begin{prop}
\label{p-conjinvariant} Any conjugacy class $\mathfrak{S}_{n}(\lambda)$
is invariant under cyclic valley-hopping.
\end{prop}

\begin{proof}
It suffices to show that, for any permutation $\pi\in\mathfrak{S}_{n}$
and $x\in[n]$, the permutation $\psi_{x}(\pi)$ has the same cycle
type as $\pi$. Because $\psi_{x}(\pi)=\pi$ whenever $x$ is a fixed
point, cyclic valley, or cyclic peak, we only need to consider the
cases when $x$ is a cyclic double ascent or cyclic double descent. 

Fix a cyclic double ascent or cyclic double descent $x$ of $\pi$.
Let us write $\pi$ as a product of cycles $\pi=C_{1}C_{2}\cdots C_{i}\cdots C_{k}$,
and let $C_{i}$ be the cycle containing $x$. First suppose that
$i<k$, i.e., $C_{i}$ is not the last cycle of $\pi$. Let $c$ denote
the first letter of $C_{i}$, and let $d$ denote the first letter
of the next cycle $C_{i+1}$. Then both $c$ and $d$ are larger than
every other element in $C_{i}$. Consider the $x$-factorization $w_{1}w_{2}xw_{4}w_{5}$
of $o(\pi)$. We observe that the letter $c$ is in $w_{1}$ and the
letter $d$ is in $w_{5}$, that all of the letters in $w_{2}xw_{4}$
are from the cycle $C_{i}$, and that each of the left-to-right maxima
of $o(\pi)$ is in $w_{1}$ or $w_{5}$. Thus, the letter $x$ is
still between the letters $c$ and $d$ in $\varphi_{x}(o(\pi))=w_{1}w_{4}xw_{2}w_{5}$,
and the left-to-right maxima of $o(\pi)$ and their positions are
unchanged after applying $\varphi_{x}$. It follows that the number
of cycles and the number of elements in each cycle of $\pi$ are unchanged
after applying $\psi_{x}$, so $\psi_{x}(\pi)$ has the same cycle
type as $\pi$. If $C_{i}$ is the last cycle of $\pi$, then a similar
argument works with $d=\infty$.
\end{proof}

\section{Results}

For a set of permutations $\Pi\subseteq\mathfrak{S}_{n}$, let 
\[
E(\Pi;t)\coloneqq\sum_{\pi\in\Pi}t^{\exc(\pi)}.
\]
Also let $\mathfrak{S}_{n,k}$ be the set of permutations in $\mathfrak{S}_{n}$
with exactly $k$ fixed points, and given $\sigma\in\mathfrak{S}_{n}$,
we let $\Orb(\sigma)=\{\,\psi_{S}(\sigma)\mid S\subseteq[n]\,\}$
denote the orbit of $\sigma$ under cyclic valley-hopping. We begin
by proving a preliminary lemma on the excedance number distribution
over a single orbit.
\begin{lem}
\label{l-excorb} Let $\sigma\in\mathfrak{S}_{n,k}$. Then
\[
\sum_{\pi\in\Orb(\sigma)}t^{\exc(\pi)}=\sum_{\pi\in\Orb(\sigma)}\frac{(s+t)^{\exc(\pi)-\cval(\pi)}(1+st)^{n-k-\cval(\pi)-\exc(\pi)}t^{\cval(\pi)}}{(1+s)^{n-k-2\cval(\pi)}}.
\]
\end{lem}

\begin{proof}
Given a fixed permutation $\sigma\in\mathfrak{S}_{n,k}$, first we
wish to prove the identity 
\begin{equation}
\Big(\sum_{\pi\in\Orb(\sigma)}t^{\exc(\pi)}\Big)(1+s)^{\cdasc(\sigma)+\cddes(\sigma)}=\sum_{\pi\in\Orb(\sigma)}(s+t)^{\cdasc(\pi)}(1+st)^{\cddes(\pi)}t^{\cval(\pi)},\label{e-mfs}
\end{equation}
which we do combinatorially by showing that the two sides of the equation
encode the same objects.

We begin with the left-hand side. Each summand in the factor $\sum_{\pi\in\Orb(\sigma)}t^{\exc(\pi)}$
corresponds to a permutation in the orbit of $\sigma$ weighted by
its excedance number. Each summand in the factor $(1+s)^{\cdasc(\sigma)+\cddes(\sigma)}$
corresponds to marking a subset of cyclic double ascents and cyclic
double descents of $\sigma$. Thus, the left-hand side counts permutations
in $\Orb(\sigma)$ where $t$ is weighting the excedance number and
$s$ is weighting the number of marked letters.

Now, let us examine the right-hand side of Equation (\ref{e-mfs}).
Each term on the right-hand side of Equation (\ref{e-mfs}) corresponds
to taking a permutation $\pi\in\Orb(\sigma)$, choosing a subset $S$
of cyclic double ascents and cyclic double descents of $\pi$, applying
$\psi_{S}$ to $\pi$, marking the letters of $S$ in $\psi_{S}(\pi)$
(which are all cyclic double ascents or cyclic double descents of
$\psi_{S}(\pi)$), and weighting the marked letters by $s$ and the
excedances of $\psi_{S}(\pi)$ by $t$. The $(s+t)^{\cdasc(\pi)}$
factor corresponds to selecting the cyclic double ascents,\footnote{More precisely, we are partitioning $\Cdasc(\pi)$ into two sets $S\cap\Cdasc(\pi)$
and $\Cdasc(\pi)\backslash S$. By Proposition \ref{p-statsets} (d),
the letters in $S\cap\Cdasc(\pi)$ are cyclic double descents of $\psi_{S}(\pi)$
and thus non-excedances of $\psi_{S}(\pi)$, so they are not given
a weight of $t$ but are given a weight of $s$ because they belong
to $S$. On the other hand, by Proposition \ref{p-statsets} (c),
the letters in $\Cdasc(\pi)\backslash S$ are cyclic double descents
of $\psi_{S}(\pi)$ and thus excedances of $\psi_{S}(\pi)$, so they
are given a weight of $t$ but not a weight of $s$ because they do
not belong to $S$.} and the $(1+st)^{\cddes(\pi)}$ factor corresponds to selecting the
cyclic double descents.\footnote{This is by similar reasoning as in the previous footnote.} 

At this point, we have accounted for all excedances of $\psi_{S}(\pi)$
which are cyclic double ascents of $\psi_{S}(\pi)$. By Equation (\ref{e-excdecompset}),
the only remaining excedances of $\psi_{S}(\pi)$ are the cyclic valleys
of $\psi_{S}(\pi)$, which are precisely the cyclic valleys of $\pi$
by Proposition \ref{p-statsets} (a); this contributes the factor
of $t^{\cval(\pi)}$. In summary, both sides of Equation (\ref{e-mfs})
count permutations in the orbit of $\sigma$ with a marked subset
$S$ of letters by the same weights, but on the right-hand side, we
are applying the involution $\psi_{S}$ to each $\pi\in\Orb(\sigma)$
before doing the counting.

Next, observe the following:
\begin{itemize}
\item By Equation (\ref{e-excdecomp}), we have 
\[
\cdasc(\pi)=\exc(\pi)-\cval(\pi).
\]
\item By Equations (\ref{e-ndecomp}) and (\ref{e-cpkcval}), we have 
\begin{align*}
\cddes(\pi) & =n-(\cval(\pi)+\cpk(\pi)+\cdasc(\pi)+k)\\
 & =n-k-\cval(\pi)-\exc(\pi).
\end{align*}
\item By the above two equations, we have 
\[
\cdasc(\sigma)+\cddes(\sigma)=n-k-2\cval(\sigma).
\]
\end{itemize}
Therefore, from (\ref{e-mfs}) we have the equation
\[
\Big(\sum_{\pi\in\Orb(\sigma)}t^{\exc(\pi)}\Big)(1+s)^{n-k-2\cval(\sigma)}=\sum_{\pi\in\Orb(\sigma)}(s+t)^{\exc(\pi)-\cval(\pi)}(1+st)^{n-k-\cval(\pi)-\exc(\pi)}t^{\cval(\pi)},
\]
and dividing both sides by $(1+s)^{n-k-2\cval(\sigma)}=(1+s)^{n-k-2\cval(\pi)}$
gives us the desired formula.
\end{proof}

\subsection{Gamma-positivity results}

Before proving our main result (Theorem \ref{t-mainthm}), we use
Lemma \ref{l-excorb} to prove a $\gamma$-positivity result for excedance
number distributions over subsets invariant under cyclic valley-hopping
and containing permutations with the same number of fixed points.
\begin{thm}
\label{t-cvalexcgpos} Let $\Pi\subseteq\mathfrak{S}_{n,k}$ be invariant
under cyclic valley-hopping. Then 
\[
E(\Pi;t)=\sum_{i=0}^{\left\lfloor (n-k)/2\right\rfloor }\gamma_{i}t^{i}(1+t)^{n-k-2i}
\]
where 
\begin{align*}
\gamma_{i} & =\left|\left\{ \,\pi\in\Pi:\cval(\pi)=i\text{ and }\cdasc(\pi)=0\,\right\} \right|\\
 & =\frac{1}{2^{n-k-2i}}\left|\left\{ \,\pi\in\Pi:\cval(\pi)=i\,\right\} \right|.
\end{align*}
\end{thm}

\begin{proof}
Taking Lemma \ref{l-excorb} and setting $s=1$ yields 
\[
\sum_{\pi\in\Orb(\sigma)}t^{\exc(\pi)}=\sum_{\pi\in\Orb(\sigma)}\frac{t^{\cval(\pi)}(1+t)^{n-k-2\cval(\pi)}}{2^{n-k-2\cval(\pi)}},
\]
and noting that $\cval(\pi)=\cval(\sigma)$ for all $\pi\in\Orb(\sigma)$\textemdash a
consequence of Proposition \ref{p-statsets} (a)\textemdash yields
\[
\sum_{\pi\in\Orb(\sigma)}t^{\exc(\pi)}=\Big(\sum_{\pi\in\Orb(\sigma)}\frac{1}{2^{n-k-2\cval(\sigma)}}\Big)t^{\cval(\sigma)}(1+t)^{n-k-2\cval(\sigma)}.
\]
Since $n-k-2\cval(\sigma)=\cdasc(\sigma)+\cddes(\sigma)$, it follows
that $\left|\Orb(\sigma)\right|=2^{n-k-2\cval(\sigma)}.$ Thus 
\[
\sum_{\pi\in\Orb(\sigma)}t^{\exc(\pi)}=t^{\cval(\sigma)}(1+t)^{n-k-2\cval(\sigma)}.
\]
Summing this equation over all orbits contributing to $\Pi$ yields
\[
E(\Pi;t)=\sum_{\pi\in\Pi}t^{\exc(\pi)}=\sum_{i=0}^{\left\lfloor (n-k)/2\right\rfloor }\gamma_{i}t^{i}(1+t)^{n-k-2i}
\]
where $\gamma_{i}$ is the number of orbits contributing to $\Pi$
containing permutations with exactly $i$ cyclic valleys. 

Since $\left|\Orb(\sigma)\right|=2^{n-k-2\cval(\sigma)}$, we have
\[
\gamma_{i}=\frac{1}{2^{n-k-2i}}\left|\left\{ \,\pi\in\Pi:\cval(\pi)=i\,\right\} \right|.
\]
Moreover, in each cyclic valley-hopping orbit there is a unique permutation
with no cyclic double ascents\textemdash this is $\psi_{S}(\sigma)$
for $S=\Cdasc(\sigma)$\textemdash so alternatively we have
\[
\gamma_{i}=\left|\left\{ \,\pi\in\Pi:\cval(\pi)=i\text{ and }\cdasc(\pi)=0\,\right\} \right|.\qedhere
\]
\end{proof}
We now give several interesting consequences of Theorem \ref{t-cvalexcgpos}.
\begin{cor}
Let $\lambda$ be a partition of $n$ with $k$ parts of size 1. Then
\[
E_{\lambda}(t)=\sum_{i=0}^{\left\lfloor (n-k)/2\right\rfloor }\gamma_{i}t^{i}(1+t)^{n-k-2i}
\]
where $\gamma_{i}=2^{-n+k+2i}\left|\left\{ \,\pi\in\mathfrak{S}_{n}(\lambda):\cval(\pi)=i\,\right\} \right|$.
\end{cor}

\begin{proof}
We know from Proposition \ref{p-conjinvariant} that $\mathfrak{S}_{n}(\lambda)$
is invariant under cyclic valley-hopping for any partition $\lambda$
of $n$. Thus the result follows from Theorem \ref{t-cvalexcgpos}
by taking $\Pi=\mathfrak{S}_{n}(\lambda)$.
\end{proof}
As mentioned in the introduction, the $\gamma$-positivity of $E_{\lambda}(t)$
follows from Brenti's formula (\ref{e-brenti}), but our approach
yields a combinatorial interpretation for the $\gamma$-coefficients.

The next corollary explains our observation in Remark \ref{r-unisym}
that the coefficient of each $s^{i}$ in $E_{\lambda}^{(\cval,\exc)}(s,t)$
appears to be a unimodal and symmetric polynomial in $t$.
\begin{cor}
\label{c-cvalexccoeff} Let $\lambda$ be a partition of $n$. Then,
for any integer $i\geq0$, the coefficient of $s^{i}$ in $E_{\lambda}^{(\cval,\exc)}(s,t)$
is a $\gamma$-positive polynomial in $t$.
\end{cor}

\begin{proof}
Let $\mathfrak{S}_{n,i}(\lambda)$ denote the set of permutations
in $\mathfrak{S}_{n}(\lambda)$ with exactly $i$ cyclic valleys.
We know that the number of cyclic valleys is constant over any cyclic
valley-hopping orbit. Thus, by the same reasoning as in the proof
of Lemma \ref{p-conjinvariant}, the set $\mathfrak{S}_{n,i}(\lambda)$
is invariant under cyclic valley-hopping, and so $E(\mathfrak{S}_{n,i}(\lambda);t)$
is $\gamma$-positive by Theorem \ref{t-cvalexcgpos}. Since 
\[
E_{\lambda}^{\cval,\exc}(s,t)=\sum_{i=0}^{\left\lfloor (n-k)/2\right\rfloor }E(\mathfrak{S}_{n,i}(\lambda);t)s^{i}
\]
(where $k$ is the number of parts of size 1 in $\lambda$), the result
follows.
\end{proof}
Let us now define $\mathfrak{S}_{n,k,i}$ to be the set of permutations
of length $n$ with exactly $k$ fixed points and $i$ cyclic valleys.
\begin{cor}
\label{c-gammakfixpts} For any $0\leq k\leq n$, we have
\[
E(\mathfrak{S}_{n,k};t)=\sum_{i=0}^{\left\lfloor (n-k)/2\right\rfloor }\frac{\left|\mathfrak{S}_{n,k,i}\right|}{2^{n-k-2i}}t^{i}(1+t)^{n-k-2i}.
\]
\end{cor}

\begin{proof}
Since $\mathfrak{S}_{n,k}$ is the union of all conjugacy classes
$\mathfrak{S}_{n}(\lambda)$ containing permutations with exactly
$k$ fixed points, each of which is invariant under cyclic valley-hopping,
it follows that $\mathfrak{S}_{n,k}$ is also invariant under cyclic
valley-hopping. Thus the result follows from Theorem \ref{t-cvalexcgpos}.
\end{proof}
The case $k=0$ of Corollary \ref{c-gammakfixpts} agrees with the
known result for derangements \cite{Shin2012,Sun2014}. The numbers
$\left|\mathfrak{S}_{n,k,i}\right|$ can be obtained via the generating
function 
\[
1+\sum_{n=1}^{\infty}\left|\mathfrak{S}_{n,k,i}\right|u^{k}t^{i}\frac{x^{n}}{n!}=\frac{\sqrt{1-t}e^{(u-1)x}}{\sqrt{1-t}\cosh(x\sqrt{1-t})-\sinh(x\sqrt{1-t})}
\]
(see \cite[Section 4.1]{Tirrell2018}).
\begin{cor}
For any $0\leq k\leq n$ and $0\leq i\leq\left\lfloor (n-k)/2\right\rfloor $,
we have
\[
E(\mathfrak{S}_{n,k,i};t)=\frac{\left|\mathfrak{S}_{n,k,i}\right|}{2^{n-k-2i}}t^{i}(1+t)^{n-k-2i}.
\]
\end{cor}

\begin{proof}
In the proof of Corollary \ref{c-cvalexccoeff}, we saw that the set
$\mathfrak{S}_{n,i}(\lambda)$ consisting of all permutations in the
conjugacy class $\mathfrak{S}_{n}(\lambda)$ with exactly $i$ cyclic
valleys is invariant under cyclic valley-hopping. Since $\mathfrak{S}_{n,k,i}$
is the union of all the sets $\mathfrak{S}_{n,i}(\lambda)$ over all
partitions $\lambda$ of $n$ with exactly $k$ parts of size 1, it
follows that $\mathfrak{S}_{n,k,i}$ is also invariant under cyclic
valley-hopping. Thus the result follows from Theorem \ref{t-cvalexcgpos}.
\end{proof}
We note that if $\Pi\subseteq\mathfrak{S}_{n}$ is invariant under
cyclic valley-hopping but contains permutations with different numbers
of fixed points, then $E(\Pi;t)$ is not necessarily $\gamma$-positive
but is a sum of $\gamma$-positive polynomials with different centers
of symmetry.

\subsection{Proof of Theorem 1 and related results}

For a set of permutations $\Pi\subseteq\mathfrak{S}_{n}$, let 
\[
E^{(\cval,\exc)}(\Pi;s,t)\coloneqq\sum_{\pi\in\Pi}s^{\cval(\pi)}t^{\exc(\pi)}.
\]
The following theorem allows us to relate the polynomials $E(\Pi;t)$
and $E^{(\cval,\exc)}(\Pi;s,t)$ whenever the subset $\Pi\subseteq\mathfrak{S}_{n,k}$
is invariant under cyclic valley-hopping. Theorem \ref{t-mainthm}
will follow as a corollary.
\begin{thm}
\label{t-cvalexc} Let $\Pi\subseteq\mathfrak{S}_{n,k}$ be invariant
under cyclic valley-hopping. Then 
\begin{equation}
E(\Pi;t)=\left(\frac{1+st}{1+s}\right)^{n-k}E^{(\cval,\exc)}\left(\Pi;\frac{(1+s)^{2}t}{(s+t)(1+st)},\frac{s+t}{1+st}\right).\label{e-cvalexc1}
\end{equation}
Equivalently, 
\begin{equation}
E^{(\cval,\exc)}(\Pi;s,t)=\left(\frac{1+u}{1+uv}\right)^{n-k}E(\Pi;v)\label{e-cvalexc2}
\end{equation}
where $u=\frac{1+t^{2}-2st-(1-t)\sqrt{(1+t)^{2}-4st}}{2(1-s)t}$ and
$v=\frac{(1+t)^{2}-2st-(1+t)\sqrt{(1+t)^{2}-4st}}{2st}.$
\end{thm}

This theorem is a cyclic analogue of a previous result by the third
author \cite[Theorem 5.1]{Zhuang2017} which relates the distribution
of $\des$ and the joint distribution of $\pk$ (the number of peaks)
and $\des$ over sets of permutations invariant under (ordinary) valley-hopping.
\begin{proof}
Taking Lemma \ref{l-excorb} and summing over all orbits contributing
to $\Pi$ yields
\begin{align*}
\sum_{\pi\in\Pi}t^{\exc(\pi)} & =\sum_{\pi\in\Pi}\frac{(s+t)^{\exc(\pi)-\cval(\pi)}(1+st)^{n-k-\cval(\pi)-\exc(\pi)}t^{\cval(\pi)}}{(1+s)^{n-k-2\cval(\pi)}}\\
 & =\left(\frac{1+st}{1+s}\right)^{n-k}\sum_{\pi\in\Pi}\left(\frac{(1+s)^{2}t}{(s+t)(1+st)}\right)^{\cval(\pi)}\left(\frac{s+t}{1+st}\right)^{\exc(\pi)}
\end{align*}
and thus Equation (\ref{e-cvalexc1}) follows. 

We obtain Equation (\ref{e-cvalexc2}) by setting $u=\frac{(1+s)^{2}t}{(s+t)(1+st)}$
and $v=\frac{s+t}{1+st}$, solving for $s$ and $t$ (which can be
done using a computer algebra system such as Maple), and simplifying.\footnote{We exchanged $u$ and $v$ with $s$ and $t$, respectively, in the
statement of Equation (\ref{e-cvalexc2}) in this theorem, so that
the $(\cval,\exc)$ polynomial would have variables $s$ and $t$
as in its definition.}
\end{proof}
Taking $\Pi=\mathfrak{S}_{n}(\lambda)$ where $\lambda=(1^{m_{1}}2^{m_{2}}\cdots)$,
we obtain from Equation (\ref{e-cvalexc2}) the formula
\[
E_{\lambda}^{(\cval,\exc)}(s,t)=\left(\frac{1+u}{1+uv}\right)^{n-m_{1}}E_{\lambda}(v)
\]
where $u$ and $v$ are defined as in Theorem \ref{t-cvalexc}. Combining
this formula with Brenti's formula (\ref{e-brenti}) proves Theorem
\ref{t-mainthm}.

Finally, we derive a formula analogous to Theorem \ref{t-mainthm}
which allows one to compute the polynomials $E_{\lambda}^{\cval}(t)$
using Eulerian polynomials. Given a set of permutations $\Pi\subseteq\mathfrak{S}_{n}$,
let 
\[
E^{\cval}(\Pi;t)\coloneqq\sum_{\pi\in\Pi}t^{\cval(\pi)}.
\]
\begin{thm}
Let $\Pi\subseteq\mathfrak{S}_{n,k}$ be invariant under cyclic valley-hopping.
Then 
\begin{equation}
E(\Pi;t)=\left(\frac{1+t}{2}\right)^{n-k}E^{\cval}\left(\Pi;\frac{4t}{(1+t)^{2}},\right).\label{e-cval1}
\end{equation}
Equivalently, 
\begin{equation}
E^{\cval}(\Pi;t)=(1+\sqrt{1-t})^{n-k}E(\Pi;w)\label{e-cval2}
\end{equation}
where $w=2t^{-1}(1-\sqrt{1-t})-1$.
\end{thm}

\begin{proof}
Equation (\ref{e-cval1}) is obtained by taking Equation (\ref{e-cvalexc1})
and setting $s=1$. Inverting Equation (\ref{e-cval1}) and simplifying
the result yields Equation (\ref{e-cval2}).
\end{proof}
\begin{thm}
\label{t-mainresult2} Let $\lambda=(1^{m_{1}}2^{m_{2}}\cdots)$ be
a partition of $n$. Then
\begin{equation}
E_{\lambda}^{\cval}(t)=\frac{n!}{z_{\lambda}}(1+\sqrt{1-t})^{n-m_{1}}\prod_{i\geq1}\left[\frac{A_{i-1}(w)}{(i-1)!}\right]^{m_{i}}\label{e-cvalexc-1}
\end{equation}
where $w=2t^{-1}(1-\sqrt{1-t})-1$.
\end{thm}

\begin{proof}
This is proven by taking Equation (\ref{e-cval2}) with $\Pi=\mathfrak{S}_{n}(\lambda)$
and combining the result with Brenti's formula (\ref{e-brenti}).
\end{proof}
\noindent \vspace{0in}

\noindent \textbf{Acknowledgements.} We thank Kyle Petersen for his
helpful feedback and an anonymous referee for pointing out a couple
mistakes on an earlier version of this manuscript.\vspace{0.3in}

\bibliographystyle{plain}
\bibliography{bibliography}

\end{document}